\documentclass[a4paper, 12pt]{amsart}

\usepackage{amsmath, amssymb, amscd, enumitem, nccmath, framed, stmaryrd, mathtools}

\usepackage{etoolbox}

\makeatletter
\let\old@tocline\@tocline
\let\section@tocline\@tocline
\newcommand{\subsection@dotsep}{4.5}
\newcommand{\subsubsection@dotsep}{4.5}
\patchcmd{\@tocline}
  {\hfil}
  {\nobreak
     \leaders\hbox{$\m@th
        \mkern \subsection@dotsep mu\hbox{.}\mkern \subsection@dotsep mu$}\hfill
     \nobreak}{}{}
\let\subsection@tocline\@tocline
\let\@tocline\old@tocline

\patchcmd{\@tocline}
  {\hfil}
  {\nobreak
     \leaders\hbox{$\m@th
        \mkern \subsubsection@dotsep mu\hbox{.}\mkern \subsubsection@dotsep mu$}\hfill
     \nobreak}{}{}
\let\subsubsection@tocline\@tocline
\let\@tocline\old@tocline

\let\old@l@subsection\l@subsection
\let\old@l@subsubsection\l@subsubsection

\def\@tocwriteb#1#2#3{%
  \begingroup
    \@xp\def\csname #2@tocline\endcsname##1##2##3##4##5##6{%
      \ifnum##1>\c@tocdepth
      \else \sbox\z@{##5\let\indentlabel\@tochangmeasure##6}\fi}%
    \csname l@#2\endcsname{#1{\csname#2name\endcsname}{\@secnumber}{}}%
  \endgroup
  \addcontentsline{toc}{#2}%
    {\protect#1{\csname#2name\endcsname}{\@secnumber}{#3}}}%

\newlength{\@tocsectionindent}
\newlength{\@tocsubsectionindent}
\newlength{\@tocsubsubsectionindent}
\newlength{\@tocsectionnumwidth}
\newlength{\@tocsubsectionnumwidth}
\newlength{\@tocsubsubsectionnumwidth}
\newcommand{\settocsectionnumwidth}[1]{\setlength{\@tocsectionnumwidth}{#1}}
\newcommand{\settocsubsectionnumwidth}[1]{\setlength{\@tocsubsectionnumwidth}{#1}}
\newcommand{\settocsubsubsectionnumwidth}[1]{\setlength{\@tocsubsubsectionnumwidth}{#1}}
\newcommand{\settocsectionindent}[1]{\setlength{\@tocsectionindent}{#1}}
\newcommand{\settocsubsectionindent}[1]{\setlength{\@tocsubsectionindent}{#1}}
\newcommand{\settocsubsubsectionindent}[1]{\setlength{\@tocsubsubsectionindent}{#1}}

\renewcommand{\l@section}{\section@tocline{1}{\@tocsectionvskip}{\@tocsectionindent}{}{\@tocsectionformat}}%
\renewcommand{\l@subsection}{\subsection@tocline{2}{\@tocsubsectionvskip}{\@tocsubsectionindent}{}{\@tocsubsectionformat}}%
\renewcommand{\l@subsubsection}{\subsubsection@tocline{3}{\@tocsubsubsectionvskip}{\@tocsubsubsectionindent}{}{\@tocsubsubsectionformat}}%
\newcommand{\@tocsectionformat}{}
\newcommand{\@tocsubsectionformat}{}
\newcommand{\@tocsubsubsectionformat}{}
\expandafter\def\csname toc@1format\endcsname{\@tocsectionformat}
\expandafter\def\csname toc@2format\endcsname{\@tocsubsectionformat}
\expandafter\def\csname toc@3format\endcsname{\@tocsubsubsectionformat}
\newcommand{\settocsectionformat}[1]{\renewcommand{\@tocsectionformat}{#1}}
\newcommand{\settocsubsectionformat}[1]{\renewcommand{\@tocsubsectionformat}{#1}}
\newcommand{\settocsubsubsectionformat}[1]{\renewcommand{\@tocsubsubsectionformat}{#1}}
\newlength{\@tocsectionvskip}
\newcommand{\settocsectionvskip}[1]{\setlength{\@tocsectionvskip}{#1}}
\newlength{\@tocsubsectionvskip}
\newcommand{\settocsubsectionvskip}[1]{\setlength{\@tocsubsectionvskip}{#1}}
\newlength{\@tocsubsubsectionvskip}
\newcommand{\settocsubsubsectionvskip}[1]{\setlength{\@tocsubsubsectionvskip}{#1}}

\patchcmd{\tocsection}{\indentlabel}{\makebox[\@tocsectionnumwidth][l]}{}{}
\patchcmd{\tocsubsection}{\indentlabel}{\makebox[\@tocsubsectionnumwidth][l]}{}{}
\patchcmd{\tocsubsubsection}{\indentlabel}{\makebox[\@tocsubsubsectionnumwidth][l]}{}{}

\newcommand{\@sectypepnumformat}{}
\renewcommand{\contentsline}[1]{%
  \expandafter\let\expandafter\@sectypepnumformat\csname @toc#1pnumformat\endcsname%
  \csname l@#1\endcsname}
\newcommand{\@tocsectionpnumformat}{}
\newcommand{\@tocsubsectionpnumformat}{}
\newcommand{\@tocsubsubsectionpnumformat}{}
\newcommand{\setsectionpnumformat}[1]{\renewcommand{\@tocsectionpnumformat}{#1}}
\newcommand{\setsubsectionpnumformat}[1]{\renewcommand{\@tocsubsectionpnumformat}{#1}}
\newcommand{\setsubsubsectionpnumformat}[1]{\renewcommand{\@tocsubsubsectionpnumformat}{#1}}
\renewcommand{\@tocpagenum}[1]{%
  \hfill {\mdseries\@sectypepnumformat #1}}

\let\oldappendix\appendix
\renewcommand{\appendix}{%
  \leavevmode\oldappendix%
  \addtocontents{toc}{%
    \protect\settowidth{\protect\@tocsectionnumwidth}{\protect\@tocsectionformat\sectionname\space}%
    \protect\addtolength{\protect\@tocsectionnumwidth}{2em}}%
}
\makeatother

\makeatletter
\settocsectionnumwidth{2em}
\settocsubsectionnumwidth{2.5em}
\settocsubsubsectionnumwidth{3em}
\settocsectionindent{1pc}%
\settocsubsectionindent{\dimexpr\@tocsectionindent+\@tocsectionnumwidth}%
\settocsubsubsectionindent{\dimexpr\@tocsubsectionindent+\@tocsubsectionnumwidth}%
\makeatother

\settocsectionvskip{10pt}
\settocsubsectionvskip{0pt}
\settocsubsubsectionvskip{0pt}
    
\settocsectionformat{\bfseries}
\settocsubsectionformat{\mdseries}
\settocsubsubsectionformat{\mdseries}
\setsectionpnumformat{\bfseries}
\setsubsectionpnumformat{\mdseries}
\setsubsubsectionpnumformat{\mdseries}

\let\oldtableofcontents\tableofcontents
\renewcommand{\tableofcontents}{%
  \vspace*{-\linespacing}
  \oldtableofcontents}

\usepackage[tikz]{mdframed}
\usepackage[all]{xy}
\usepackage{multirow}

\usepackage{leftidx}

\newtheorem{thm}{Theorem}[section] 
\newtheorem{lem}[thm]{Lemma}

\newtheorem{claim}[thm]{Claim}

\newcommand{\G}{\mathbb{G}}

\newcommand{\rmin}{{\rm min}}

\newcommand{\Aut}{{\rm Aut}}

\newcommand{\Hom}{{\rm Hom}}

\newcommand{\ds}{\displaystyle}
\newcommand{\Exp}{{\rm Exp}}
\newcommand{\GL}{{\rm GL}}
\newcommand{\id}{{\rm id}}
\newcommand{\inn}{{\rm inn}}

\newcommand{\Mat}{{\rm Mat}}

\newcommand{\Pic}{{\rm Pic}\,}

\newcommand{\red}{{\rm red}}

\newcommand{\Hopf}{{\rm Hopf}}
\newcommand{\alggr}{{\rm alg.gr}}
\newcommand{\gr}{{\rm gr}}

\newcommand{\PGL}{{\rm PGL}}

\newcommand{\A}{\mathbb{A}}
\newcommand{\E}{\mathbb{E}}
\newcommand{\cA}{\mathcal{A}}
\newcommand{\cB}{\mathcal{B}}
\newcommand{\cC}{\mathcal{C}}

\newcommand{\cN}{\mathcal{N}}

\newcommand{\Z}{\mathbb{Z}}

\newcommand{\circled}[1]{\raise0.1ex\hbox{\textcircled{\scriptsize{\raise0.2ex\hbox{#1}}}}}
\newcommand{\ouparrow}{\raise0.1ex\hbox{\textcircled{\scriptsize{\raise0.1ex\hbox{$\uparrow$}}}}}
\newcommand{\orightarrow}{ \raise0.1ex\hbox{\textcircled{\scriptsize{\raise0.23ex\hbox{\hspace{0.1mm}$\rightarrow$}}}} }

\allowdisplaybreaks[4]

\makeatletter

\renewcommand\section{\@startsection{section}{1}%
  \z@{-.5\linespacing\@plus-.7\linespacing}{.5\linespacing}%
  {\bf \Large {\normalfont\scshape}}}

\renewcommand\subsection{\@startsection{subsection}{2}%
  \z@{-.5\linespacing\@plus-.7\linespacing}{.5\linespacing}%
  {\bf \large {\normalfont\scshape}}}

\renewcommand\subsubsection{\@startsection{subsubsection}{3}%
  \z@{.5\linespacing\@plus.7\linespacing}{.5\linespacing}
   {\bf  {\normalfont\scshape}}}
\makeatother

\makeatletter
\@namedef{subjclassname}{{\rm 2020} Mathematics Subject Classification}
\makeatother

\begin{document}

\title{
Exponential matrices,  $\G_a$-actions on projective spaces 
and modular representations of elementary abelian $p$-groups
}
\subjclass[2020]{Primary 15A54, Secondary 14L30, 20C20}
\keywords{Matrix theory, Algebraic geometry, Modular representation theory}
\author[Ryuji Tanimoto]{Ryuji Tanimoto}
\address{Faculty of Education, Shizuoka University, 836 Ohya, Suruga-ku, Shizuoka 422-8529, Japan} 
\email{tanimoto.ryuji@shizuoka.ac.jp}
\maketitle

\begin{abstract}
Let $k$ be an algebraically closed field of characteristic $p \, (\geq 0)$ and 
let $\G_a$ denote the additive group of $k$. 
In this article, we reveal the following: 

In the first, we show that there exist one-to-one correspondences among 
the following sets {\rm (a)}, {\rm (b)}, {\rm (c)}, {\rm (d)}, {\rm (e)}, {\rm (f)}: 
\begin{enumerate}[label = {\rm (\alph*)}]
\item the set $\Mat(n, k[T])^E$ of all exponential matrices of $\Mat(n, k[T])$; 

\item the set of all homomorphisms $h : k[\, \GL(n, k) \,] \to k[T]$ of Hopf algebras; 

\item the set of all homomorphisms $\varphi : \G_a \to \GL(n, k)$ of algebraic groups; 

\item the set of all homomorphisms $\theta : \G_a \to \PGL(n, k)$ of algebraic groups; 

\item the set of all homomorphisms $\vartheta : \G_a \to \Aut(\mathbb{P}^{n - 1})$ of algebraic groups; 

\item the set of all $\G_a$-actions $\mu$ on $\mathbb{P}^{n - 1}$. 
\end{enumerate} 
These correspondences are given up to equivalences.

In the second, assuming $p > 0$, we build a relationship between 
modular representations of elementary abelian $p$-groups 
and exponential matrices. 
We show that there exists a one-to-one correspondence between 
the following sets ($\alpha$) and ($\beta$): 
\begin{enumerate}

\item[($\alpha$)] the set of all group homomorphisms from $(\Z/p\Z)^r$ to $\GL(n, k)$, where $r$ ranges over all non-negative integers. 

\item[($\beta$)] the set $\Mat(n, k[T])^E \times \Z_{\geq 0}$. 
\end{enumerate}
\end{abstract}

%

\section{Introduction}

In this article, we investigate relationships among exponential matrices, 
$\G_a$-actions on projective spaces and 
modular representations of elementary abelian $p$-groups.  
These three objects belong to different reseach areas. 
At present, there seems no interactive result among these objects. 
We are in the following research situations: 

In characteristi zero, exponential matrices are classified (see \cite{Tanimoto 2019}). 
In positive characteristic, exponential matrices 
of size $n$-by-$n$ for $1 \leq n \leq 5$ are described (see \cite{Tanimoto 2019, Tanimoto 2020}).

In characteristic zero, 
$\G_a$-actions on the $n$-dimensional projective space are classified (see \cite{Gurjar-Masuda-Miyanishi}). 
Perhaps, in positive characteristic,  
we need to discover $p$-techniques for exploring $\G_a$-actions on projective spaces.

Modular representations of elemetary abelian $p$-groups are 
wild except for few cases (see \cite{Bondarenko-Drozd}). 
We refer to \cite{Benson} about various results of 
modular representations of elementary abelian $p$-groups.

Toward collaborative developments of 
exponential matrices, 
$\G_a$-actions on projective spaces 
and modular representations of elementary abelian $p$-groups, 
we build relationships among these objects. 
Our main results are the following Theorems 0.1 and 0.2, 
where $k$ is an algebraically closed field of characteristic $p \, (\geq 0)$ 
and $\G_a$ denotes the additive group of $k$:

\begin{thm} 
There exist one-to-one correspondences among 
the following sets {\rm (a)}, {\rm (b)}, {\rm (c)}, {\rm (d)}, {\rm (e)}, {\rm (f)}: 
\begin{enumerate}[label = {\rm (\alph*)}]
\item the set $\Mat(n, k[T])^E$ of all exponential matrices of $\Mat(n, k[T])$; 

\item the set of all homomorphisms $h : k[\, \GL(n, k) \,] \to k[T]$ of Hopf algebras; 

\item the set of all homomorphisms $\varphi : \G_a \to \GL(n, k)$ of algebraic groups; 

\item the set of all homomorphisms $\theta : \G_a \to \PGL(n, k)$ of algebraic groups; 

\item the set of all homomorphisms $\vartheta : \G_a \to \Aut(\mathbb{P}^{n - 1})$ of algebraic groups; 

\item the set of all $\G_a$-actions $\mu$ on $\mathbb{P}^{n - 1}$. 
\end{enumerate} 
These correspondences are given up to equivalences.
\end{thm}

The following theorem works provided that $p > 0$. 

\begin{thm}
Assume $p > 0$. 
There exists a one-to-one correspondence between the following sets {\rm ($\alpha$)} and {\rm ($\beta$)}: 
\begin{enumerate}

\item[\rm ($\alpha$)] the set of all group homomorphisms from $(\Z/p\Z)^r$ to $\GL(n, k)$, where $r$ ranges over all non-negative integers; 

\item[\rm ($\beta$)] the set $\Mat(n, k[T])^E \times \Z_{\geq 0}$. 
\end{enumerate}
\end{thm}
\vspace{1em}

We employ the following notation and a definition:

For any commutative ring $R$, we denote by $\Mat(n, R)$ the set of 
all $n \times n$ matrices whose entries belong to $R$ 
and denote by $R[T]$ the polynomial ring in one variable over $R$. 

For any field $k$, a matrix $A(T)$ of $\Mat(n, k[T])$ 
is said to be an {\it exponential matrix} of $\Mat(n, k[T])$ 
if $A(T)$ satisfies the following conditions {\rm (1)} and {\rm (2)}: 
\begin{enumerate}[label = {\rm (\arabic*)}]
\item $A(T) \, A(T') = A(T + T')$ in $\Mat(n, k[T, T'])$, 
where $k[T, T']$ denotes the polynomial ring in two variables over $k$. 

\item $A(0) = I_n$. 
\end{enumerate}
We denote by $\Mat(n, k[T])^E$ the set of all exponential matrices 
$A(T)$ of $\Mat(n, k[T])$.

%
%
%
%
%
%

\section{Correspondences}

\subsection{Equivalences}

\begin{enumerate}[label = {\rm (\alph*)}]
\item {\bf Equivalence of exponential matrices of $\Mat(n, k[T])$. } 
\vspace{3pt}

For any regular matrix $P$ of $\GL(n, k)$, we can define a map 
$\inn_P^{\rm (a)} : \Mat(n, k[T])^E \to \Mat(n, k[T])^E$ as 
\[
 \inn_P^{\rm (a)}(\, A(T) \,) := P \, A(T) \, P^{-1} . 
\]

Let $A_1(T)$ and $A_2(T)$ be exponential matrices of $\Mat(n, k[T])$. 
We say that $A_1(T)$ and $A_2(T)$ are {\it equivalent} if there exists a regular matrix 
$P$ of $\GL(n, k)$ such that 
\[
A_2(T) =   \inn_P^{\rm (a)} ( \,  A_1(T) \, ). 
\]

\item {\bf Equivalence of homomorphisms $k[\, \GL(n, k) \,] \to k[T]$ of Hopf algebras. } 
\vspace{3pt} 

Let  
\[
 \bigl( \, k[\, \GL(n, k) \,], \; \Delta_{k[\, \GL(n, k) \,]}, \; \epsilon_{k[\, \GL(n, k) \,]}, \; S_{k[\, \GL(n, k) \,]} \, \bigr)
\]
be the commutative Hopf algebra corresponding to the linear algebraic group $\GL(n, k)$, where  
$k[\, \GL(n, k) \,]$ denotes the coordinate ring of $\GL(n, k)$, i.e., 
\[
 k[\, \GL(n, k) \,] = k\bigl[\, x_{i, j} \; (\, 1 \leq i, j \leq n \,) , \; 1/ d \, \bigr] , \quad 
d : = \det\bigl( \, (x_{i, j})_{1 \leq i, j \leq n} \, \bigr), 
\]
and where $\Delta_{k[\, \GL(n, k) \,]}$ is a comultiplication, $\epsilon_{k[\, \GL(n, k) \,]}$ is a counit and $S_{k[\, \GL(n, k) \,]}$ is an antipode, 
i.e., 
\begin{align*}
 \Delta_{k[\, \GL(n, k) \,]}(x_{i, j}) 
 & = \sum_{\ell = 1}^{n} x_{i, \ell} \otimes x_{\ell, j} 
\qquad \text{ for all \quad $1 \leq i, j \leq n$} ,  \\ 
\epsilon_{k[\, \GL(n, k) \,]}(x_{i, j}) 
 & = 
\left\{
\begin{array}{l @{\qquad} l}
 1 & \text{ if \quad $i = j$}, \\
 0 & \text{ if \quad $i \ne j$} , 
\end{array}
\right.  \\
 S_{k[\, \GL(n, k) \,]}(x_{i, j}) 
 & = (-1)^{i + j} \, d^{-1} \det\bigl( \, (x_{r, s})_{r \ne j, \; s \ne i}) \, \bigr) 
\qquad \text{ for all \quad $1 \leq i, j \leq n$} . 
\end{align*}

Let $(k[T], \Delta_{k[T]}, \epsilon_{k[T]}, S_{k[T]})$ be the Hopf algebra corresponding to the additive group $\G_a$, 
where $\Delta_{k[T]}$ is a comultiplication, $\epsilon_{k[T]}$ is a counit and $S_{k[T]}$ is an antipode. 
So, $\Delta_{k[T]}(T) = T \otimes 1 + 1 \otimes T$, $\epsilon_{k[T]}(T) = 0$ and $S_{k[T]}(T) = - T$.

For any regular matrix $P$ of $\GL(n, k)$, we denote by $\widehat{p}_{i, j}$ the $(i, j)$-th entry of $P^{-1}$ 
$(1 \leq i, j \leq n)$, i.e., $P^{-1} = (\, \widehat{p}_{i, j} \,)_{1 \leq i, j \leq n}$. 
We can define a homomorphism $\inn_P^{\rm (b)} : k[\, \GL(n, k) \,] \to k[\, \GL(n, k) \,]$ as 
\[
 \inn_P^{\rm (b)}(x_{i, j} )
 = 
\sum_{1 \leq \lambda, \mu \leq n} p_{i, \, \lambda} \cdot x_{\lambda, \, \mu} 
 \cdot \widehat{p}_{\mu, \, j}
\qquad 
\text{ for all \quad $1 \leq i, j \leq n$} . 
\]

Let $h_1 : k[\, \GL(n, k) \,] \to k[T]$ and $h_2 : k[\, \GL(n, k) \,] \to k[T]$ be homomorphisms 
of Hopf algebras. 
We say that $h_1$ and $h_2$ are {\it equivalent} if there exists a regular 
matrix $P$ of $\GL(n, k)$ such that 
\[
 h_2 = h_1 \circ \inn_P^{\rm (b)} .
\]

\item {\bf Equivalence of homomorphisms $\G_a \to \GL(n, k)$ of algebraic groups.} 
\vspace{3pt}

For any regular matrix $P$ of $\GL(n, k)$, 
we can define a homomorphism $\inn_P^{\rm (c)} : \GL(n, k) \to \GL(n, k)$ as 
\[
 \inn_P^{\rm (c)}( Q ) := P \, Q \, P^{-1}. 
\]

Let $\varphi_1 : \G_a \to \GL(n, k)$ and $\varphi_2 : \G_a \to \GL(n, k)$ be 
homomorphisms of algebraic groups. 
We say that $\varphi_1$ and $\varphi_2$ are {\it equivalent} if there exists a regular 
matrix $P$ of $\GL(n, k)$ such that 
\[
 \varphi_2 = \inn_P^{\rm (c)} \circ \varphi_1 . 
\]

\item {\bf Equivalence of homomorphisms $\G_a \to \PGL(n, k)$ of algebraic groups.} 
\vspace{3pt}

For any regular matrix $P$ of $\GL(n, k)$, 
we can define a homomorphism $\inn_P^{\rm (d)} : \PGL(n, k) \to \PGL(n, k)$ as 
\[
 \inn_P^{\rm (d)}( [Q] ) := [\, P \, Q \, P^{-1} \,] . 
\]

Let $\theta_1 : \G_a \to \PGL(n, k)$ and $\theta_2 : \G_a \to \PGL(n, k)$ 
be homomorphisms of algebraic groups. 
We say that $\theta_1$ and $\theta_2$ are {\it equivalent} if there exists a 
regular matrix $P$ of $\GL(n, k)$ such that 
\[
 \theta_2 = \inn_P^{\rm (d)} \circ \theta_1 . 
\]

\item {\bf Equivalence of homomorphisms $\G_a \to \Aut(\mathbb{P}^{n - 1})$ of algebraic groups.} 
\vspace{3pt}

For any $\sigma \in \Aut(\mathbb{P}^{n - 1})$, we can define a homomorphism 
$\inn_\sigma^{\rm (e)} : \Aut(\mathbb{P}^{n - 1}) \to \Aut(\mathbb{P}^{n - 1})$ as 
\[
 \inn_\sigma^{\rm (e)}(f) := \sigma \circ f \circ \sigma^{-1} . 
\]

Let $\vartheta_1 : \G_a \to \Aut(\mathbb{P}^{n - 1})$ and $\vartheta_2 : \G_a \to \Aut(\mathbb{P}^{n - 1})$ be homomorphisms of algebraic groups. 
We say that $\vartheta_1$ and $\vartheta_2$ are {\it equivalent} 
if there exists an element $\sigma$ of $\Aut(\mathbb{P}^{n - 1})$
such that $\vartheta_2 = \inn_\sigma^{\rm (e)} \circ  \vartheta_1$.

\item {\bf Equivalence of $\G_a$-actions on $\mathbb{P}^{n - 1}$.} 
\vspace{3pt}

Two $\G_a$-actions $\mu_1 : \G_a \times \mathbb{P}^{n - 1} \to \mathbb{P}^{n - 1}$ and 
$\mu_2 : \G_a \times \mathbb{P}^{n - 1} \to \mathbb{P}^{n - 1}$ are 
said to be {\it equivalent} 
if there exists an automorphism $\sigma : \mathbb{P}^{n - 1} \to \mathbb{P}^{n - 1}$ such that 
$\mu_2 \circ (\id_{\G_a} \times \sigma) = \sigma \circ {\mu_1}$, i.e., the following diagram is commutative, where $\id_{\G_a} : \G_a \to \G_a$ denotes the identity map: 
\[
\xymatrix@R=36pt@C=36pt@M=6pt{
 \G_a \times \mathbb{P}^{n - 1} \ar[r]^(.57){\mu_1} \ar[d]_{\id_{\G_a} \times \sigma}  & \mathbb{P}^{n - 1} \ar[d]^\sigma \\
 \G_a \times \mathbb{P}^{n - 1} \ar[r]^(.57){\mu_2}  & \mathbb{P}^{n - 1}  
}
\]
\end{enumerate}

\subsection{Proof of Theorem 0.1}

\subsubsection{(a) $\cong$ (b) $\cong$ (c)}

We shall give one-to-one correspondences among sets (a), (b), (c).

We denote by $\cA^+$ the set of all matrices $A(T)$ of $\Mat(n, k[T])$ 
such that $\det A(T) \in k \backslash \{ \, 0 \,\}$, i.e., 
\begin{fleqn}[8em]
\begin{align*}
\cA^+
 & := 
\{ \, 
 A(T) \in \Mat(n, k[T]) 
\mid 
\det  A(T) \in k \backslash \{ 0 \}
\, \} . 
\end{align*}
\end{fleqn}
Clearly, $\Mat(n, k[T])^E \subset \cA^+$ (see \cite[Lemmas 1.8 and 1.9]{Tanimoto 2019}).

We denote by $\cB^+$ the set of all homomorphisms $h : k[\, \GL(n, k) \,] \to k[T]$ of $k$-algebras, 
i.e., 
\begin{fleqn}[8em]
\begin{align*}
 \cB^+ 
 & :=  \Hom_{\text{$k$-alg}} (k[\, \GL(n, k) \,], k[T])  . 
\end{align*}
\end{fleqn}
Clearly, $\Hom_\Hopf(k[\, \GL(n, k) \,], k[T]) \subset \cB^+$, where $\Hom_\Hopf$ denotes 
homomorphisms of Hopf algebras.

We denote by $\cC^+$ the set of all homomorphisms $\varphi : \G_a \to \GL(n, k[T])$ 
of affine varieties, i.e., 
\begin{fleqn}[8em]
\begin{align*}
 \cC^+ 
 & := \Hom_{\text{aff.var}}(\G_a, \GL(n, k)) . 
\end{align*}
\end{fleqn}
Clearly, $\Hom_\alggr(\G_a, \GL(n, k)) \subset \cC^+$, where $\Hom_\alggr$ denotes 
homomorphisms of algebraic groups.

Let $F_{\cB^+, \, \cA^+} : \cA^+ \to \cB^+$ be the map defined, as follows: 
To a matrix $A(T) = (\, a_{i, j}(T) \,)_{1 \leq i, j \leq n}$ of $\cA^+$, 
we assign an element $h : k[\, \GL(n, k) \,] \to k[T]$ of $\cB^+$ as 
\[
 h(x_{i, j}) := a_{i, j}(T) \qquad \text{ for all \quad $1 \leq i, j \leq n$}  . 
\]

Let $F_{\cC^+, \, \cB^+} : \cB^+ \to \cC^+$ be the map defined, as follows: 
To an element $h : k[\, \GL(n, k) \,] \to k[T]$ of $\cB^+$, we naturally 
assign an element $\varphi : \G_a \to \GL(n, k)$ of $\cC^+$.

We have the following diagram:
\[
\xymatrix@R=36pt@C=36pt@M=6pt{
 \cA^+ \ar[r]^{F_{\cB^+, \, \cA^+}}
  & \cB^+  \ar[r]^{F_{\cC^+, \, \cB^+}}
  & \cC^+ 
\\
 \Mat(n, k[T])^E  \ar@{^(->}[u] 
  & \Hom_\Hopf(k[\, \GL(n, k) \,], k[T])  \ar@{^(->}[u]
  & \Hom_\alggr(\G_a, \GL(n, k))  \ar@{^(->}[u]
}
\]

\begin{lem}
The following assertions {\rm (1)} and {\rm (2)} hold true: 
\begin{enumerate}[label = {\rm (\arabic*)}]
\item The map $F_{\cB^+, \, \cA^+}$ gives a one-to-one coresspondence 
between the sets $\cA^+$ and $\cB^+$. 

\item The map $F_{\cC^+, \,  \cB^+}$ gives a one-to-one correspondence 
between the sets $\cB^+$ and $\cC^+$. 
\end{enumerate} 
\end{lem}

\begin{proof}
The proofs of assertions (1) and (2) are straightforward. 
\end{proof}

\begin{lem}
Let 
\[
 A(T) = (\, a_{i, j}(T) \,)_{1 \leq i, j \leq n} \in \cA^+, \qquad 
 h := F_{\cB^+, \,  \cA^+}(A(T)) \in \cB^+ , \qquad  
 \varphi := F_{\cC^+, \,  \cB^+}(h) \in \cC^+ . 
\] 
Then the following conditions {\rm (1), (2), (3), (4)} are equivalent: 
\begin{enumerate}[label = {\rm (\arabic*)}]
\item $A(T) \in \Mat(n, k[T])^E$.

\item We have 
\begin{align*}
 & a_{i, j}(T \otimes 1 + 1 \otimes T) = \sum_{\ell = 1}^{n} a_{i, \ell}(T) \otimes a_{\ell, j}(T) 
\qquad 
\text{ for all \quad $1 \leq i, j \leq n$} 
\intertext{and}
&
 a_{i, j}(0) 
 = 
\left\{
\begin{array}{l @{\qquad} l}
 1 & \text{ if \quad $i = j$}, \\
 0 & \text{ if \quad $i \ne j$}. 
\end{array}
\right.
\end{align*} 

\item $h \in \Hom_\Hopf(k[\, \GL(n, k) \,], k[T])$. That is to say, the following two diagrams 
are commutative: 
\begin{align*}
\xymatrix@R=36pt@C=24pt@M=6pt{
 k[\, \GL(n, k) \,] \otimes_k k[\, \GL(n, k) \,]  \ar[rr]^(.6){h \otimes h} 
 & 
 & k[T] \otimes_k k[T] \\
k[\, \GL(n, k) \,] \ar[rr]_h \ar[u]^{\Delta_{k[\, \GL(n, k) \,]}} 
 &
 & k[T] \ar[u]_{\Delta_{k[T]}}
}
\end{align*}

\begin{align*}
\hspace{18pt}
\xymatrix@R=36pt@C=6pt@M=6pt{
k[\, \GL(n, k) \,] \ar[rrrrrrr]^h \ar[rrrd]_{\epsilon_{k[\, \GL(n, k) \,]}} 
 &  
 &
 &
 & 
 &
 & 
 & k[T] \ar[lllld]^{\epsilon_{k[T]}}   \\
 &
 &
 & k  
 &
 & 
 &
 & 
}
\end{align*}

\item $\varphi \in \Hom_\alggr(\G_a, \GL(n, k))$. 
\end{enumerate} 
\end{lem}

\begin{proof}
We first prove (1) $\Longleftrightarrow$ (2).  
Condition (1) holds true if and only if 
\begin{align*}
 & a_{i, j}(T + T') = \sum_{\ell = 1}^{n} a_{i, \ell}(T) \cdot a_{\ell, j}(T') 
\qquad 
\text{ for all \quad $1 \leq i, j \leq n$} 
\intertext{and}
&
 a_{i, j}(0) 
 = 
\left\{
\begin{array}{l @{\qquad} l}
 1 & \text{ if \quad $i = j$}, \\
 0 & \text{ if \quad $i \ne j$}. 
\end{array}
\right.
\end{align*} 
Consider the $k$-algebra isomorphsim $\xi : k[T, T'] \to k[T] \otimes_k k[T]$ 
defined by  
\[
 \xi(T) := T \otimes 1, \qquad 
 \xi(T') := 1 \otimes T. 
\]
The above two conditions hold true 
if and only if 
\begin{align*}
 & a_{i, j}(T \otimes 1 + 1 \otimes T) = \sum_{\ell = 1}^{n} a_{i, \ell}(T) \otimes a_{\ell, j}(T) 
\qquad 
\text{ for all \quad $1 \leq i, j \leq n$} 
\intertext{and}
&
 a_{i, j}(0) 
 = 
\left\{
\begin{array}{l @{\qquad} l}
 1 & \text{ if \quad $i = j$}, \\
 0 & \text{ if \quad $i \ne j$}. 
\end{array}
\right.
\end{align*}

We next prove (2) $\Longleftrightarrow$ (3). 
For all $1 \leq i, j \leq n$, we have 
\begin{align*}
 \bigl( \,  (\, h \otimes h \,) \circ \Delta_{k[\, \GL(n, k) \,]} \, \bigr) (x_{i, j} )
 & = 
 (\, h \otimes h \,) \left( \; 
\sum_{\ell = 1}^{n} x_{i, \ell} \otimes x_{\ell, j} 
\; \right) 
 = \sum_{\ell = 1}^{n} a_{i, \ell}(T) \otimes a_{\ell, j}(T) , 
\\
(\, \Delta_{k[T]} \circ h \, )(x_{i, j}) 
 & =  \Delta_{k[T]} (a_{i, j}(T)) 
 = a_{i, j}(T \otimes 1 + 1 \otimes T) , \\
 \bigl( \, \epsilon_{k[T]} \circ h \, \bigr) (x_{i, j})
 & = \epsilon_{k[T]} (\, a_{i, j}(T) \,) 
 = a_{i, j}(0) , 
\\
\epsilon_{k[\, \GL(n, k) \,]}( x_{i, j} ) 
& = 
\left\{
\begin{array}{l @{\qquad} l}
 1 & \text{ if \quad $i = j$}, \\
 0 & \text{ if \quad $i \ne j$} . 
\end{array}
\right. 
\end{align*} 
So, the implication (3) $\Longrightarrow$ (2) is clear. 
Conversely, if condition (2) holds true, then we have 
$( (h\otimes h) \circ \Delta_{k[\, \GL(n, k) \,]})(x_{i, j}) = (\Delta_{k[T]} \circ h)(x_{i, j})$ 
and $(\epsilon_{k[T]} \circ h)( x_{i, j} ) = \epsilon_{k[\, \GL(n, k) \,]}(x_{i, j})$ 
for all $1 \leq i, j \leq n  + 1$, 
and thereby have 
\begin{align*}
 \bigl(\, (\, h \otimes h \,) \circ \Delta_{k[\, \GL(n, k) \,]} \, \bigr) \left( \frac{1}{d} \right) 
 & = \frac{1}{\bigl( \, (\, h \otimes h \,) \circ \Delta_{k[\, \GL(n, k) \,]} \, \bigr) (d)} 
 = \frac{1}{ (\, \Delta_{k[T]} \circ h \,)(d) }
 = (\, \Delta_{k[T]} \circ h \,)\left( \frac{1}{d} \right) , \\
 (\, \epsilon_{k[T]} \circ h \,)\left( \frac{1}{d} \right) 
 & = 
\frac{1}{ (\, \epsilon_{k[T]} \circ h \,)(d) }
 = \frac{1}{ \epsilon_{k[\, \GL(n, k) \,]}(d) } 
 = \epsilon_{k[\, \GL(n, k) \,]}\left( \frac{1}{d} \right) . 
\end{align*} 
Thus condition (3) holds true.

The equivalence (3) $\Longleftrightarrow$ (4) is clear. 
\end{proof}

\begin{lem}
The following assertions {\rm (1)} and {\rm (2)} hold true: 
\begin{enumerate}[label = {\rm (\arabic*)}]
\item The map $F_{\cB^+, \,  \cA^+}$ gives a one-to-one correspondence 
between the sets $\Mat(n, k[T])^E$ and $\Hom_\Hopf(k[\, \GL(n, k) \,], k[T])$. 

\item The map $F_{\cC^+, \,  \cB^+}$ gives a one-to-one correspondence 
between the sets $\Hom_\Hopf(k[\, \GL(n, k) \,], k[T])$ and $\Hom_\alggr(\G_a, \GL(n, k))$. 
\end{enumerate} 
\end{lem}

\begin{proof}
See Lemmas 1.1 and 1.2. 
\end{proof}

\begin{lem}
For each $i = 1, 2$, let $A_i(T)$ be an exponential matrix of $\Mat(n, k[T])^E$, 
let $h_i : k[\, \GL(n, k) \,] \to k[T]$ be a homomorphism of Hopf algebras such 
that 
\[
 h_i := F_{\cB^+, \,  \cA^+}(A_i(T)), 
\] 
and let $\varphi_i : \G_a \to \GL(n, k)$ be a homomorphism of 
algebraic groups over $k$ such that 
\[
 \varphi_i := F_{\cC^+, \,  \cB^+}(h_i(T)) . 
\]
Then the following conditions {\rm (1)}, {\rm (2)}, {\rm (3)} are equivalent: 
\begin{enumerate}[label = {\rm (\arabic*)}]
\item $A_1(T)$ and $A_2(T)$ are equivalent. 

\item $h_1$ and $h_2$ are equivalent. 

\item $\varphi_1$ and $\varphi_2$ are equivalent. 
\end{enumerate} 
\end{lem}

\begin{proof}
We first prove (1) $\Longleftrightarrow$ (2). 
Let $P = (\, p_{i, j} \,)_{1 \leq i, j \leq n}$ be a regular matrix of $\GL(n, k
)$. 
Write $P^{-1} = (\, \widehat{p}_{i, j} \,)_{1 \leq i, j \leq n}$, 
$A_1(T) = (\, a_{i, j}^{(1)} \,)_{1 \leq i, j \leq n}$ and 
$A_2(T) = (\, a_{i, j}^{(2)} \,)_{1 \leq i, j \leq n}$. 
So, $A_2(T) = \inn_P^{\rm (a)} \circ A_1(T)$ if and only if 
\[ 
 a_{i, j}^{(2)} 
=
 \sum_{1 \leq \lambda, \mu \leq n} 
p_{i, \, \lambda} \cdot a_{\lambda, \, \mu}^{(1)} \cdot \widehat{p}_{\mu, \, j} 
\qquad 
\text{ for all \quad $1 \leq i,  j \leq n$} . 
\]
We have 
\begin{align*}
& h_2(x_{i, j} ) = a_{i, j}^{(2)}  
\intertext{and}
& h_1(\, \inn_P^{\rm (b)} (x_{i, j}) \,)
  = h_1 \left( \; 
\sum_{1 \leq \lambda, \mu \leq n} 
p_{i, \, \lambda} \cdot x_{\lambda, \, \mu} \cdot \widehat{p}_{\mu, \, j} 
 \; \right) 
 = \sum_{1 \leq \lambda, \mu \leq n} 
p_{i, \, \lambda} \cdot a_{\lambda, \, \mu}^{(1)} \cdot \widehat{p}_{\mu, \, j} \\
& 
\hspace{23em} \text{ for all \quad $1 \leq i, j \leq n$} . 
\end{align*}
Thus $A_2(T) = \inn_P^{\rm (a)} ( A_1(T) ) $ if and only if 
$h_2(x_{i, j}) = ( h_1 \circ \inn_P^{\rm (b)}  )(x_{i, j}) $ for all $1 \leq i, j \leq n$. 
This equivalence implies the equivalence (1) $\Longleftrightarrow$ (2).

The proof of equivalence  (2) $\Longleftrightarrow$ (3) is straightforward. 
%
\end{proof}

\subsubsection{(c) $\cong$ (d)}

We shall give a one-to-one correspondence between sets (c) and (d).

Let $\pi : \GL(n, k) \to \PGL(n, k)$ be the natural surjection. 
Clearly, $\pi$ is a homomorphism of algebraic groups.

\begin{lem}
\label{1.6}
For any homomorphism $\theta : \G_a \to \PGL(n, k)$, 
there exists a unique homomorphism $\varphi : \G_a \to \GL(n, k)$ such that $\pi \circ \varphi = \theta$, 
i.e., the following diagram is commutative: 
\[
\xymatrix@C=36pt{
                                               & \GL(n, k) \ar[d]^{\pi} \\
\G_a \ar[r]_(.3){\theta}  \ar[ru] ^{\varphi}  & \PGL(n, k) 
}
\]
\end{lem}

\begin{proof}
If $\theta$ is trivial, i.e., $\theta(\G_a) = \{\, e_{\PGL(n, k)} \,\}$, we are done. 
So, assume that $\theta$ is nontrivial. 
The homomorphism $\theta$ factors into a composite of homomorphisms 
\[
\xymatrix@R=36pt@C=36pt@M=6pt{
 \G_a \ar[r]^q & X \ar[r]^(.33)i & \PGL(n, k)  
}
\]
with $q$ a faithfully flat morphism and $i$ a closed embedding (see \cite[Theorem 5.39]{Milne}). 
The image $X = \theta(\G_a)_\red$ is smooth (see \cite[Proposition 1.28 and Corollary 5.26]{Milne}). 
Thus $X$ is isomorphic to the affine line, i.e., $X \cong \A^1$. 
We consider the inverse image $E := \pi^{-1}(X)_\red$ of $X$ under $\pi$. 
So, $\pi|_E : E \to X$ is a $\G_m$-bundle over $X$, 
where $\pi|_E$ is the restriction of $\pi$ to $E$. 
Note that $\pi|_E$ is a trivial $\G_m$-bundle over $X$. 
In fact, there exist a line bundle $\pi' : L \to X$  and an open embedding $\iota : E \to L$ so that $\pi' \circ \iota = \pi|_E$. 
Since $\Pic(X) = 0$, the line bundle $\pi' : L\to X$ is trivial, and thereby $\pi|_E$ is a trivial 
$\G_m$-bundle over $X$. 
We have the following commutative diagram:
\[
\xymatrix@C=60pt@R=36pt@M=8pt{
 E \ar[d]_{\pi|_E} \ar@{^(->}[r]^\iota & L \ar[ld]_(.5){\pi'}  \ar[r]^(.4)\cong & X \times \A^1 \ar[lld]^{(x, a) \mapsto x}  \\
  X
}
\]
Since $E \cong X \times \G_m$, there exists a section $s$ of $\pi|_E$. 
We can define a morphism $f : \G_a \to \GL(n, k)$ as
\[
f := \iota_E \circ s \circ q , 
\]  
where $\iota_E : E \to \GL(n, k)$ is the natural closed immersion. 
Clearly, 
\begin{align}
\tag{\ref{1.6}.1}
\pi \circ f = \theta . 
\end{align} 
Let $\nu : \GL(n, k) \to \GL(n, k)$ be the homomorphism defined by  
\[
 \nu(A) := \frac{1}{\det A} \, A . 
\]
Clearly, 
\begin{align}
\tag{\ref{1.6}.2}
\pi \circ \nu = \pi . 
\end{align} 
Since $\theta$ is a homomorphism, 
we have 
\begin{align}
\tag{\ref{1.6}.3}
m_{\PGL(n, k)} \circ (\theta \times \theta) = \theta \circ m_{\G_a}, 
\end{align}
i.e., 
the following diagram is commutative: 
\[
\xymatrix@R=36pt@C=36pt@M=6pt{
 \G_a \times \G_a \ar[d]_{m_{\G_a}}  \ar[r]^(.35){\theta \times \theta}
 & \PGL(n, k) \times \PGL(n, k) \ar[d]^{m_{\PGL(n, k)}}  \\
 \G_a \ar[r]^\theta & \PGL(n, k) 
}
\]
We have 
\[
\pi \circ m_{\GL(n, k)} \circ  (\nu \times \nu ) \circ (f \times f) 
 = \pi \circ \nu \circ f \circ m_{\G_a} 
\]
since 
\begin{align*}
\begin{array}{@{\quad} r c l}
\multicolumn{3}{l}{ 
 \pi \circ m_{\GL(n, k)} \circ  (\nu \times \nu ) \circ (f \times f) 
} \\
 & = & m_{\PGL(n, k)} \circ (\pi \times \pi) \circ (\nu \times \nu ) \circ (f \times f) \\
 & \overset{(\ref{1.6}.2)}{=} & m_{\PGL(n, k)} \circ (\pi \times \pi) \circ (f \times f) \\ 
 & \overset{(\ref{1.6}.1)}{=} & m_{\PGL(n, k)} \circ (\theta \times \theta)  \\ 
 & \overset{(\ref{1.6}.3)}{=} & \theta \circ m_{\G_a} \\
 & \overset{(\ref{1.6}.1)}{=} & \pi \circ f \circ m_{\G_a} \\ 
 & \overset{(\ref{1.6}.2)}{=} & \pi \circ \nu \circ f \circ m_{\G_a} . 
\end{array}
\end{align*} 
Now, we have 
\[
  m_{\GL(n, k)} \circ  (\nu \times \nu ) \circ (f \times f) =  \nu \circ f \circ m_{\G_a} . 
\]
Letting $\varphi := \nu \circ f $, we have a homomorphism $\varphi : \G_a \to \GL(n, k)$ so that 
$\pi \circ \varphi = \theta$.

Let $\varphi : \G_a \to \GL(n, k)$  and let $\phi : \G_a \to \GL(n, k)$ 
be homomorphisms so that $\pi \circ \varphi = \theta$ and $\pi \circ \phi = \theta$. 
So, we have $\pi \circ \nu \circ \varphi = \pi \circ \nu \circ \phi$, 
which implies $\nu \circ \varphi = \nu \circ \phi$. 
We can show $\nu \circ \varphi = \varphi$ and $\nu \circ \phi = \phi$ 
since $\det \varphi(t) = \det \phi(t) = 1$ for all $t \in \G_a$ 
(see \cite[Lemmas 1.8 and 1.9]{Tanimoto 2019}). 
Thus $\varphi = \phi$. 
\end{proof}

We can define a map $\varpi :  \Hom_\alggr(\G_a, \GL(n, k)) \to \Hom_\alggr(\G_a, \PGL(n, k) )$ as 
\[
\varpi(\varphi) := \pi \circ \varphi . 
\]

\begin{lem}
The map $\varpi : \Hom_\alggr(\G_a, \GL(n, k)) \to \Hom_\alggr(\G_a, \PGL(n, k))$  is bijective. 
\end{lem}

\begin{proof}
See Lemma 1.5. 
\end{proof}

\begin{lem}
Let $\theta_i : \G_a \to \PGL(n, k)$ $(i = 1, 2)$ be homomorphisms 
and let $\varphi_i : \G_a \to \GL(n, k)$ $(i = 1, 2)$ be homomorphisms such that $\theta_i =  \pi \circ \varphi_i $ for $i = 1, 2$. 
Then the following conditions {\rm (1)} and {\rm (2)} are equivalent: 
\begin{enumerate}[label = {\rm (\arabic*)}]
\item $\theta_1$ and $\theta_2$ are equivalent. 

\item $\varphi_1$ and $\varphi_2$ are equivalent. 
\end{enumerate}
\end{lem}

\begin{proof}
Let $P \in \PGL(n, k)$ and $Q \in \GL(n, k)$ satisfy $P = \pi(Q)$. 
Clearly, $\inn_P^{\rm (d)}  \circ \pi = \pi \circ \inn_Q^{\rm (c)} $, i.e., the following diagram is commutative: 
\[
\xymatrix@C=36pt@R=36pt@M=6pt{
 \GL(n, k) \ar[r]^{\inn_{Q}^{\rm (c)} } \ar[d]_\pi & \GL(n, k) \ar[d]^\pi \\
 \PGL(n, k) \ar[r]_{\inn_P^{\rm (d)} }   & \PGL(n, k) 
}
\]
If $\theta_2 = \inn_P^{\rm (d)}  \circ \theta_1$, then 
\[
\pi \circ \varphi_2
 = \theta_2 
 = \inn_P^{\rm (d)}  \circ \theta_1
 = \inn_P^{\rm (d)}  \circ \pi \circ \varphi_1 
 = \pi \circ \inn_{Q}^{\rm (c)}  \circ \varphi_1  , 
\]
which implies 
$\varphi_2 = \inn_{Q}^{\rm (c)}  \circ \varphi_1 $ (see Lemma 1.6). 
Conversely, if $\varphi_2 = \inn_{Q}^{\rm (c)}  \circ \varphi_1$, it is clear 
that $\theta_2 = \inn_P^{\rm (d)}  \circ \theta_1$.

%
\end{proof}

\subsubsection{(d) $\cong$ (e)}

We shall give a one-to-one correspondence between sets (d) and (e).

We can define an isomorphism 
$\jmath : \PGL(n, k) \to \Aut(\mathbb{P}^{n - 1})$ of algebraic groups as  
\[
 P
\quad \mapsto \quad  
\bigl(\; 
 (x_0 : x_1 : \cdots : x_{n - 1}) 
 \; \mapsto \; 
 (x_0 : x_1 : \cdots : x_{n - 1}) 
\, {^t}Q 
\; \bigr) ,  
\]
where $P = \pi(Q)$ and $\pi : \GL(n, k) \to \PGL(n, k)$ is the natural surjection.

\begin{lem}
The map $\Hom_\alggr(\G_a, \PGL(n, k)) \to \Hom_\alggr(\G_a, \Aut(\mathbb{P}^{n - 1}))$ 
defined by $\theta \mapsto \jmath \circ \theta$ is a bijection. 
\end{lem}

\begin{proof}
The proof is straightforward. 
\end{proof}

\begin{lem}
Let $\theta_i$ $( i = 1, 2)$ be elements of $ \Hom_\alggr(\G_a, \PGL(n, k))$  
and let $\vartheta_i := \jmath \circ \theta_i$ $(i = 1, 2)$ be elements of 
$\Hom_\alggr(\G_a, \Aut(\mathbb{P}^{n - 1}))$. 
Then the following conditions {\rm (1)} and {\rm (2)} are equivalent: 
\begin{enumerate}[label = {\rm (\arabic*)}]
\item $\theta_1$ and $\theta_2$ are equivalent. 

\item $\vartheta_1$ and $\vartheta_2$ are equivalent. 
\end{enumerate}
\end{lem}

\begin{proof} 
For any element $P \in \PGL(n, k)$, we have 
$
 \inn^{\rm (e)}_{\jmath(P)}
= 
\jmath  \circ \inn^{\rm (d)} _P \circ \jmath^{-1} 
$. 
So, the following diagram is commutative: 
\begin{align*}
\xymatrix@C=36pt@R=12pt@M=6pt{
  \PGL(n, k)   \ar[r]^{\jmath} \ar[dd]_{\inn^{\rm (d)} _P} 
 & \Aut(\mathbb{P}^{n - 1})  \ar[dd]^{ \inn^{\rm (e)} _{\jmath(P)} }  \\
  & \\
   \PGL(n, k)  \ar[r]^{\jmath}
 &  \Aut(\mathbb{P}^{n - 1}) 
}
\end{align*} 
This diagram implies the equivalence (1) $\Longleftrightarrow$ (2). 
\end{proof}

\subsubsection{(e) $\cong$ (f)}

We shall give a one-to-one correspondence between sets (e) and (f).

We can define a map $\Phi : \Hom_\alggr(\G_a, \Aut(\mathbb{P}^{n - 1}) ) \to \{\, \text{$\G_a$-actions on $\mathbb{P}^{n - 1}$} \,\} $ as 
\[
 \Phi(\vartheta) := \mu_\vartheta ,  
 \qquad 
\mu_\vartheta(t, x) := \vartheta(t) (x) . 
\] 
So, the following diagram is commutative, where $\mu_{\mathbb{P}^{n - 1}}(f, x) := f(x)$, 
where $\id_{\mathbb{P}^{n - 1}} : \mathbb{P}^{n - 1} \to \mathbb{P}^{n - 1}$ denotes  
the identity map: 
\[
\xymatrix@R=36pt@C=36pt@M=6pt{
 \G_a \times \mathbb{P}^{n - 1} \ar[r]^{\mu_\theta} 
\ar[d]_{\vartheta \times \id_{\mathbb{P}^{n - 1}}} & \mathbb{P}^{n - 1} \\
  \Aut( \mathbb{P}^{n - 1} ) \times \mathbb{P}^{n - 1} \ar[ru]_{\mu_{\mathbb{P}^{n - 1}}} & 
}
\]
%
%

We can define a map $\Psi :  \{\, \text{$\G_a$-actions on 
$\mathbb{P}^{n - 1}$} \,\}  \to   \Hom_\alggr(\G_a, \Aut(\mathbb{P}^{n - 1}) )$ as 
\[
 \Psi(\mu) := \vartheta_\mu ,  
 \qquad \vartheta_\mu(t) := (\, x \mapsto \mu(t, x) \,) \in \Aut(\mathbb{P}^{n - 1}) . 
\]

\begin{lem}
$\Psi \circ \Phi = \id$ and $\Phi \circ \Psi = \id$. 
\end{lem}

\begin{proof}
The proof is straightforward. 
\end{proof}

\begin{lem} 
Let $\vartheta_i $ $(i = 1, 2)$ be elements of 
$\Hom_\alggr(\G_a, \Aut(\mathbb{P}^{n - 1}))$. 
For any $\sigma \in \Aut(\mathbb{P}^{n - 1})$, 
the following conditions {\rm (1)} and {\rm (2)} are equivalent: 
\begin{enumerate}[label = {\rm (\arabic*)}]
\item $\vartheta_2 = \inn_\sigma^{\rm (e)}  \circ \vartheta_1$. 

\item The following diagram is commutative: 
\[
\xymatrix@R=36pt@C=36pt@M=6pt{
 \G_a \times \mathbb{P}^{n - 1} \ar[r]^(.43){\vartheta_1 \times \id_{\mathbb{P}^{n - 1}}} \ar[d]_{\id_{\G_a} \times \sigma}
 & \Aut(\mathbb{P}^{n - 1}) \times \mathbb{P}^{n - 1} \ar[r]^(.68){\mu_{\mathbb{P}^{n - 1}}} 
 & \mathbb{P}^{n - 1} \ar[d]^\sigma \\
 \G_a \times \mathbb{P}^{n - 1} \ar[r]^(.43){\vartheta_2 \times \id_{\mathbb{P}^{n - 1}}} 
 & \Aut(\mathbb{P}^{n - 1}) \times \mathbb{P}^{n - 1} \ar[r]^(.68){\mu_{\mathbb{P}^{n - 1}}}
 & \mathbb{P}^{n - 1} 
}
\]
\end{enumerate} 
\end{lem}

\begin{proof}
The proof is straightforward. 
\end{proof}

\begin{lem} 
Let $\vartheta_i $ $(i = 1, 2)$ be elements of 
$\Hom_\alggr(\G_a, \Aut(\mathbb{P}^{n - 1}))$. 
Let $\mu_i := \Phi(\vartheta_i)$ $(i = 1, 2)$ be $\G_a$-actions on $\mathbb{P}^{n - 1}$. 
Then the following conditions {\rm (1)} and {\rm (2)} are equivalent: 
\begin{enumerate}[label = {\rm (\arabic*)}]
\item $\mu_1$ and $\mu_2$ are equivalent. 

\item $\vartheta_1$ and $\vartheta_2$ are equivalent. 
\end{enumerate} 
\end{lem}

\begin{proof}
The implication (1) $\Longrightarrow$ (2) follows from the following diagram: 
\[
\xymatrix@R=84pt@C=36pt@M=6pt{
 \G_a \times \mathbb{P}^{n - 1} \ar[r]^(.43){ \vartheta_1 \times \id_{\mathbb{P}^{n - 1}} } 
\ar@/_36pt/[rr]_{\mu_1} \ar[d]_{\id_{\G_a} \times \sigma}
 & 
 \Aut( \mathbb{P}^{n - 1} ) \times \mathbb{P}^{n - 1} \ar[r]^(.65){\mu_{\mathbb{P}^{n - 1}}} 
  & \mathbb{P}^{n - 1} \ar[d]^\sigma \\
 \G_a \times \mathbb{P}^{n - 1} \ar[r]_(.43){ \vartheta_2 \times \id_{\mathbb{P}^{n - 1}} } 
\ar@/^36pt/[rr]^{\mu_2} 
 & 
 \Aut( \mathbb{P}^{n - 1} ) \times \mathbb{P}^{n - 1} \ar[r]_(.65){\mu_{\mathbb{P}^{n - 1}}} 
  & \mathbb{P}^{n - 1} 
}
\]

The implication (2) $\Longrightarrow$ (1) follows from the following diagram: 
\[
\xymatrix@R=48pt@C=36pt@M=6pt{
 \G_a \times \mathbb{P}^{n - 1} \ar[r]^(.43){ \vartheta_1 \times \id_{\mathbb{P}^{n - 1}} } 
\ar@/^36pt/[rr]^{\mu_1} \ar[d]_{\id_{\G_a} \times \sigma}
 & 
 \Aut( \mathbb{P}^{n - 1} ) \times \mathbb{P}^{n - 1} \ar[r]^(.65){\mu_{\mathbb{P}^{n - 1}}} 
  & \mathbb{P}^{n - 1} \ar[d]^\sigma \\
 \G_a \times \mathbb{P}^{n - 1} \ar[r]_(.43){ \vartheta_2 \times \id_{\mathbb{P}^{n - 1}} } 
\ar@/_36pt/[rr]_{\mu_2} 
 & 
 \Aut( \mathbb{P}^{n - 1} ) \times \mathbb{P}^{n - 1} \ar[r]_(.65){\mu_{\mathbb{P}^{n - 1}}} 
  & \mathbb{P}^{n - 1} 
}
\]
\end{proof}

\section{Modular representations of elementary abelian $p$-groups and exponential matrices}

From now on, we assume that the characteristic $p$ of $k$ is positive.

\subsection{Preparations for proving Theorem 0.2}

\subsubsection{$\E_r(n, k)$, $\E_{\geq 0}(n, k)$}

For any $r \geq 0$ and $n \geq 1$, we denote by $\E_r(n, k)$ 
the set of all group homomorphisms from $(\Z/p\Z)^r$ to $\GL(n, k)$ 
and denote by $\E_{\geq 0}(n, k)$ the direct sum of $\E_r(n, k)$, where $r$ ranges over all non-negative integers, 
i.e.,  
\begin{align*}
\renewcommand{\arraystretch}{1.5}
\left\{
\begin{array}{r @{\,} l}
 \E_r(n, k) &:= \Hom_\gr((\Z/p\Z)^r, \GL(n, k)) \qquad \text{ for all \quad $r \geq 0$} , \\
 \E_{\geq 0}(n, k) & := \ds \coprod_{r \geq 0} \E_r(n, k). 
 \end{array}
 \right. 
\end{align*} 
Clearly, $\E_0(n, k)$ consists of one element, which we denote by $\rho_0$.

\subsubsection{$\cN_r(n, k)$, $\rho_{N_1, \ldots, N_r}$}

For any $r \geq 1$ and $n \geq 1$, we let 
\begin{align*}
\cN_r(n, k) &:= 
\left\{\, (N_1, \ldots, N_r) \in \Mat(n, k)^r \; \left| \;
\begin{array}{l}
\text{$N_i^p = O_n$ for all $1 \leq i \leq r$} , \\
\text{$N_i \, N_j = N_j \, N_i$ for all $1 \leq i, j \leq r$}
\end{array} 
\right. \, \right\} . 
\end{align*}
For any element $(N_1, \ldots, N_r)$ of $\cN_r(n, k)$, 
we can define an element $\rho_{N_1, \ldots, N_r}$ of $\E_r(n, k)$ as 
\[
\rho_{N_1, \ldots, N_r}(a_1, \ldots, a_r) 
:= \prod_{i = 1}^r (I_n + N_i)^{a_i} . 
\]

\begin{lem}
Assume $r \geq 1$. The map $\cN_r(n, k) \to \E_r(n, k)$ defined by 
$(N_1, \ldots, N_r) \mapsto \rho_{N_1, \ldots, N_r}$ is a bijection. 
%
\end{lem}

\begin{proof}
The proof is straightforward. 
\end{proof}

\subsubsection{$\pi_r$, $\pi$}

For any $r \geq 0$, we can define a map 
\[
\pi_r : \E_r(n, k) 
 \to 
 \Mat(n, k[T])^E 
\]
by separating the cases $r = 0$ and $r \geq 1$. 
In the case where $r = 0$, we let $\pi_0(\rho_0) := I_n$. 
In the case where $r \geq 1$, for any element $\rho$ of $\E_r(n, k)$, 
there exists a unique element $(N_1, \ldots, N_r)$ of $\cN_r(n, k)$ such that $\rho = \rho_{N_1, \ldots, N_r}$. 
Let $A_\rho(T)$ be the matrix of $\Mat(n, k[T])^E$ defined by  
\[
 A_\rho(T) := \prod_{i = 1}^r \Exp(T^{p^{i - 1}} \, N_i) . 
\]
Clearly, $A_\rho(T)$ is an exponential matrix. 
Thus we can define $\pi_r(\rho)$ as 
\[
 \pi_r(\rho) := A_\rho(T) .
\]

Now, let $\pi : \E_{\geq 0}(n, k) \to \Mat(n, k[T])^E$ be the map defined by 
\[
\pi (\rho) := \pi_r(\rho) \qquad \text{ if \quad $\rho \in \E_r(n, k)$}  . 
\]

\begin{lem}
$\pi$ is surjective. 
\end{lem}

\begin{proof}

Let $A(T)$ be an exponential matrix of $\Mat(n, k[T])^E$. 
We can express $A(T)$ as 
\[
 A(T) = \prod_{i = 1}^r \Exp(T^{p^{i - 1}} \, N_i) 
\]
for some $r \geq 1$ and $(N_1, \ldots, N_r) \in \cN_r(n, k)$ (see \cite[Lemma 1.5]{Tanimoto 2019}). 
So, $\rho_{N_1, \ldots , N_r} \in \E_r(n, k)$ and 
\begin{align*}
\pi(\rho_{N_1, \ldots , N_r}) 
 =  
\pi_r(\rho_{N_1, \ldots , N_r}) = A_{\rho_{N_1, \ldots , N_r}}(T) = A(T) . 
\end{align*}
\end{proof}

\subsubsection{$l(\rho)$, $\rho_{\rm min}$}

Let $\rho \in \E_{\geq 0}(n, k)$. 
If $\rho \in \E_r(n, k)$ for some $r \geq 1$, 
there exists a unique element $(N_1, \ldots, N_r)$ of $\cN_r(n, k)$ 
such that $\rho = \rho_{N_1, \ldots, N_r}$. 
Let 
\[
 l(\rho)
  := 
  \left\{
  \begin{array}{l @{\qquad} l}
  \min\{\,
  i \in \Z_{\geq 0} \mid  
 \, \text{$N_j = O_n$ for all $j \geq i + 1$} \, \} 
  & \text{ if \quad $r \geq 1$} , \\
 0 
  & \text{ if \quad $r = 0 $} . 
  \end{array}
  \right. 
\]
We can define a homomorphism $\rho_\rmin : (\Z/p\Z)^{l(\rho)} \to \GL(n, k)$ as 
\[
 \rho_\rmin := 
 \left\{
 \begin{array}{l @{\qquad} l}
 \rho_{N_1, \ldots, N_{l(\rho)}} & \text{ if \quad $l(\rho) \geq 1$}, \\
 \rho_0 & \text{ if \quad $l(\rho) = 0$} . 
 \end{array}
\right.
\]
Clearly, the equality $l(\rho) = 0$ holds true if and only if $\rho$ is trivial.

\begin{lem}
For any $\rho \in \E_{\geq 0}(n, k)$, we have 
\[
 \pi(\rho) = \pi(\rho_\rmin) . 
\]
\end{lem}

\begin{proof}
If $l(\rho) \geq 1$, we can express $\rho$ as 
$\rho = \rho_{N_1, \ldots, N_r}$ for some $r \geq 1$ and $(N_1, \ldots, N_r) \in \cN_r(n, k)$. 
We have $(N_1, \ldots, N_r) = (N_1, \ldots, N_{l(\rho)}, O_n, \ldots , O_n )$, 
which implies 
\[
 \pi(\rho) 
 = A_{\rho_{N_1, \ldots, N_r}}(T) 
 = \prod_{i = 1}^r \Exp(T^{p^{i - 1}} \, N_i) 
 =  \prod_{i = 1}^{l(\rho)} \Exp(T^{p^{i - 1}} \, N_i) 
 = A_{\rho_{N_1, \ldots, N_{l(\rho)}}}(T) 
 = \pi(\rho_\rmin) . 
\]
If $l(\rho) = 0$, then $\rho$ is trivial. So, $\pi(\rho) = I_n$. 
Since $\rho_\rmin = \rho_0$, we have $\pi(\rho_\rmin) = I_n$. 
\end{proof}

\subsection{Proof of Theorem 0.2}

Let 
\[
\E_{\rmin}(n, k) := \{ \, \rho \in \E_{\geq 0}(n, k) \mid \rho = \rho_\rmin \,\} . 
\]
Clearly, $\rho_0 \in \E_\rmin(n, k)$. 
Let $\pi|_{\E_{\rmin}(n, k)} : \E_{\rmin}(n, k) \to \Mat(n, k[T])^E$ be the restriction 
of $\pi : \E_{\geq 0}(n, k) \to \Mat(n, k[T])^E$ to the subset $\E_{\rmin}(n, k)$ of $\E_{\geq 0}(n, k)$. 

\begin{claim}
$\pi|_{\E_{\rmin}(n, k)}$ is a bijection. 
\end{claim}

\begin{proof}
The map $\pi|_{\E_{\rmin}(n, k)}$ is surjective (see Lemmas 2.2 and 2.3). 

We shall show that $\pi|_{\E_{\rmin}(n, k)}$ is injective by separating two cases. 
Let $\rho_1, \rho_2 \in \E_{\rmin}(n, k)$ and assume $\pi(\rho_1) = \pi(\rho_2)$. 
\medskip

\noindent
{\bf Case $l(\rho_1) = 0$.} 
In this case, $\rho_1$ is trivial and thereby $\pi(\rho_1) = I_n$. 
So, $\pi(\rho_2) = I_n$, and thereby $\rho_2$ is trivial. 
Thus $\rho_1 = \rho_2$. 
\medskip 

\noindent 
{\bf Case $l(\rho_1) \geq 1$.} 
From the above case, we know $l(\rho_2) \geq 1$. 
Let $r := l(\rho_1)$ and $s := l(\rho_2)$. 
So, we can express $\rho_1$ and $\rho_2$ as 
\begin{align*}
\left\{
\begin{array}{l @{\qquad} l} 
 \rho_1 = \rho_{N_1, \ldots, N_r}  & \text{ for some \;  $(N_1, \ldots, N_r) \in \cN_r(n, k)$} , \\
 \rho_2 = \rho_{M_1, \ldots, M_s} & \text{ for some \;  $(M_1, \ldots, M_s) \in \cN_s(n, k)$} .  
\end{array}
\right. 
\end{align*}
We may assume $r \geq s$. 
Since $A_{\rho_1}(T) = \pi(\rho_1) = \pi(\rho_2) =  A_{\rho_2}(T)$, we have 
\[
(N_1, \ldots, N_r) = (M_1, \ldots, M_s, O_n, \ldots, O_n) . 
\] 
Since $N_r \ne O_n$, we must have $r = s$.  
Thus $\rho_1 = \rho_2$. 
\end{proof}

For $i \geq 1$, we denote by $\E_{\rmin}^{+i}(n, k)$ the set of all homomorphisms 
\[
\rho_{N_1, \ldots, N_{r + i}} : (\Z/p\Z)^{r + i} \to \GL(n, k) , \qquad
  (N_1, \ldots, N_{r + i}) \in \cN_{r + i}(n, k) 
\] 
such that the following conditions (1) and (2) hold true: 
\begin{enumerate}[label = {\rm (\arabic*)}]
\item $\rho_{N_1, \ldots, N_r} \in \E_{\rmin}(n, k)$. 

\item $N_{r + j} = O_n$ for all $1 \leq j \leq i$. 
\end{enumerate}

\begin{claim}
We have 
\[
\E_{\geq 0}(n, k) = \coprod_{i \geq 0} \E_{\rmin}^{+i}(n, k) , 
\] 
where $\E_\rmin^{+0}(n, k) := \E_\rmin(n, k)$. 
\end{claim}

\begin{proof}
Choose any element $\rho$ of $\E_{\geq 0}(n, k)$. 
If $\rho \in \E_r(n, k)$ for some $r \geq 1$, 
we can express $\rho$ as $\rho = \rho_{N_1, \ldots, N_r}$ for some $(N_1, \ldots, N_r) \in \cN_r(n, k)$. 
Let $i := r - l(\rho)$. We have $\rho \in \E_{\rmin}^{+i}(n, k)$. Thus $\E_{\geq 0}(n, k) \subset \bigcup_{i \geq 0} \E_{\rmin}^{+i}(n, k)$. 
Clearly, $\E_{\rmin}^{+i}(n, k) \cap \E_{\rmin}^{+j}(n, k)= \emptyset$ 
for all integers $i, j \geq 0$ with $i \ne j$. 
\end{proof}

\begin{claim}
There exists a natural one-to-one correspondence between the sets $\E_{\geq 0}(n, k)$ and $\E_{\rmin}(n, k) \times \Z_{\geq 0}$, 
i.e.,  
\[
\E_{\geq 0}(n, k) \cong \E_{\rmin}(n, k) \times \Z_{\geq 0} . 
\] 
\end{claim}

\begin{proof}
For any $i \geq 1$, there exists a natural one-to-one correspondence between 
the sets $\E_{\rmin}^{+i}(n, k)$ and $\E_{\rmin}(n, k)$.
Thus we have the desired bijection (see Claim 2.5). 

\end{proof}

We can obtain from Claims 2.6 and 2.4 the natural bijection 
\[
 \E_{\geq 0}(n, k) \cong \Mat(n, k[T])^E \times \Z_{\geq 0} . 
\]
Now, we complete the proof of Theorem 0.2. 
\vspace{1em}

From the above proof, we have the following commutative diagram, 
where $p_1 : \Mat(n, k[T])^E \times \Z_{\geq 0} \to \Mat(n, k[T])^E$ is 
the projection onto the first component: 
\[
\xymatrix@R=36pt@C=36pt@M=6pt{
 \E_{\geq 0}(n, k) \ar[r]^(.4)\cong \ar[d]_\pi 
 & \Mat(n, k[T])^E \times \Z_{\geq 0}  \ar[ld]^{p_1} \\
 \Mat(n, k[T])^E
  & 
}
\]
We can conclude that $\pi$ is a trivial $\Z_{\geq 0}$-fibration.

\end{document}